\documentclass[reqno]{amsart}
\usepackage{amsmath,amssymb,cite,color}
\usepackage[mathscr]{euscript}
\usepackage[english]{babel}
\usepackage{orcidlink}
\theoremstyle{plain}
\newtheorem{theorem}{Theorem}[section]

\newtheorem{lemma}[theorem]{Lemma}

\theoremstyle{definition}

\theoremstyle{remark}

\numberwithin{equation}{section}
\numberwithin{theorem}{section}

\newcommand{\mc}[1]{{\mathcal #1}}

\newcommand{\bb}[1]{{\mathbb #1}}

\newcommand{\rme}{\mathrm{e}}

\newcommand{\rmd}{\mathrm{d}}

\def\const{{\rm const.}\,}

\newcommand{\supp}{\mathop{\rm supp}\nolimits}

\title[Vorticity confinement in an infinite cylinder]{Vorticity confinement for 2D incompressible flows in an infinite cylinder}

\author[P.\ Butt\`a]{Paolo Butt\`a \orcidlink{0000-0002-3193-4282}}
\address{Dipartimento di Matematica\\
Sapienza Universit\`a di Roma\\
P.le Aldo Moro 5, 00185 Roma\\
Italy}

\email{butta@mat.uniroma1.it}

\author[G.\ Cavallaro]{Guido Cavallaro \orcidlink{0000-0001-5668-370X}}
\address{Dipartimento di Matematica\\
Sapienza Universit\`a di Roma\\
P.le Aldo Moro 5, 00185 Roma\\
Italy}

\email{cavallar@mat.uniroma1.it}



\date{}

\begin{document}

\begin{abstract}
We study the confinement of vorticity for two-dimensional incompressible flows in an infinite cylinder. For Navier-Stokes solutions with non-negative and compactly supported initial vorticity, we derive quantitative decay estimates showing that the vorticity mass outside regions whose distance from the initial support grows like $\sqrt{t\log^\alpha t}$ (with $\alpha>1$) or like $t^\beta$ (with $\beta>1/2$) becomes, respectively, super-polynomially or stretched-exponentially small. The analysis combines an iterative scheme with an antisymmetry property of the Biot-Savart kernel. In the Euler case, by coupling this approach with a first-moment estimate from \cite{CD19}, we recover the confinement bound of \cite{CD19} and refine it slightly: the diameter of the vorticity support grows at most like $(t\log t)^{1/3}$, rather than $t^{1/3}\log^2 t$. %%(Commun. Math. Phys., 367, 1077--1093, 2019)
\end{abstract}

\maketitle

\section{Introduction and main result}
\label{sec:1}

We study a two-dimensional incompressible fluid of unit density on the infinite cylinder $\bb R \times \bb T$, where $\bb T$ denotes the circle of unit radius. This is equivalent to describing the fluid on the strip
\begin{equation}
\label{strip}
S := \left\{ x = (x_1, x_2) \in \bb R^2 \colon x_1 \in \bb R, , x_2 \in [0, 2\pi] \right\},
\end{equation}
with periodic boundary conditions in the $x_2$-direction, that is, motions in $\bb R^2$ whose velocity field $u = (u_1, u_2)$ is $2\pi$-periodic in $x_2$. 

Consider first the case of viscous fluids. In this setting, we work with motions $u \in C^0([0,+\infty); \mc C_\mathrm{u}(S))$, where $\mc C_\mathrm{u}(S)$ denotes the space of bounded, uniformly continuous functions $v \colon S \to \bb R^2$ satisfying the periodicity condition $v(x_1,0) = v(x_1,2\pi)$. Existence and uniqueness of mild solutions to the incompressible Navier–Stokes equations in this class are known to hold for divergence-free initial data $u_0 \in \mc C(S)$ (see, for instance, \cite{Ga15} and references therein).

Denoting by $\omega = \partial_1 u_2 - \partial_2 u_1$ the vorticity, as explained, e.g., in \cite[Sec.~2]{Ga15}, if we decompose
\begin{equation}
\label{vel_field}
u(x,t) = m(x_1,t) + \hat u(x,t) \quad \text{ where } m(x_1,t) = \int_0^{2\pi}\!\rmd x_2\, u(x,t),
\end{equation}
then $m(x_1,t) = (m_1, m_2(x_1,t))$, where the constant $m_1$ can be set to zero through an appropriate Galilean transformation. Moreover, $\partial_1 m_2(x_1,t) = \int_0^{2\pi}\!\rmd x_2\, \omega(x,t)$, and the oscillatory component $\hat u(x,t)$ can be expressed in terms of the vorticity via the cylindrical Biot–Savart law,
\begin{equation}
\label{hatu}
\hat u(x,t) = \int_S\! \rmd y\, \nabla_x^\perp G(x,y)\,[\omega(y,t)-\partial_1 m_2(y_1,t)],
\end{equation}
where $\nabla^\perp := (\partial_2, -\partial_1)$ and
\begin{equation}
\label{G}
G(x, y) = -\frac12 \log \left[ \cosh(x_1 - y_1) - \cos(x_2 - y_2) \right]
\end{equation}
is the Green function for the Poisson equation on the domain $S$ with periodic boundary conditions. We note that the component $m_2(x_1,t)$, which cannot be fully recovered from the vorticity, will play no role in our analysis, since only the first component of the velocity field is relevant in the proof.

The vorticity evolution is governed by the equation
\begin{gather}
\label{eq:omega}
\begin{cases}
\partial_t\omega (x,t) + (u\cdot\nabla)\omega(x,t) = \nu\Delta \omega, \\
\omega(x,0) = \omega_0(x),
\end{cases}
\end{gather}
where $\omega_0 = \partial_1 u^0_2 - \partial_2 u^0_1$ and $\nu$ is the viscosity. For later purposes, we recall here the following weak formulation of Eq.~\eqref{eq:omega}. In view of the incompressibility condition $\nabla \cdot u(x,t)=0$, given any smooth function $f(x,t)$ with compact support in $S\times [0,+\infty)$, if
\begin{equation}
\label{Aver}
\omega_t[f] := \int_S\! \rmd x\,\omega (x,t) f(x,t) 
\end{equation}
then $t\mapsto \omega_t[f]$ solves
\begin{equation}
\label{Weakeq}
\frac{\rmd}{\rmd t}\omega_t[f] = \omega_t [u \cdot \nabla f] + \omega_t[\partial_t f] + \nu\,\omega_t[\Delta f].
\end{equation}

We aim to study the spreading of an initial bounded, non‑negative vorticity with compact support. Although the vorticity instantaneously diffuses throughout the entire strip, one can still quantify the extent of the region in which the bulk of the vorticity remains concentrated and derive upper bounds on its growth in time. Specifically, we will show that, for any sufficiently large positive time, the integral of the vorticity over the region $|x_1| > \sqrt{t \log^\alpha t}$, for any $\alpha > 1$, is bounded above by $t^{-\ell}$ for every $\ell > 0$; whereas in the region $|x_1| > t^\beta$, for any $\beta > 1/2$, this integral is smaller than $\exp(-t^\delta)$ for every $0 < \delta < 2\beta - 1$. These results are stated precisely in Theorem~\ref{teo:mt} below and are analogous to those obtained in the whole plane; see \cite[Sec.~4]{Mar94}.

It is natural to compare this problem with the corresponding one for the Euler equations, in which the vorticity is simply transported  ($\nu=0$). When the domain is the whole plane and the initial vorticity has compact support, one can investigate how the diameter of the support grows in time. The conservation laws of total vorticity mass, center of vorticity, and moment of inertia, defined respectively by
\[
m := \int\!\rmd x\, \omega , \quad B := \frac 1m \int\!\rmd x\, x\, \omega, \quad I = \int\!\rmd x\,  |x-B|^2 \omega,
\]
allow one to show that the diameter of the support does not exceed $\const (t \log t)^{1/4}$ \cite{ISG} (the weaker estimate $\const t^{1/3}$ was previously proved in \cite{Mar94}). More recently \cite{CD19}, this question has been addressed for the two-dimensional Euler equations in an infinite cylinder, which is precisely the domain considered in the present paper. The authors take a non‑negative initial vorticity compactly supported in the cylinder and show that for $t \ge 1$ the diameter of the support is bounded above by $\const t^{1/3}\log^2 t$. In this setting, the total vorticity mass and the horizontal component of the center of vorticity are conserved\footnote{It is worth emphasizing that these two quantities retain their conserved character in the Navier–Stokes flow($\nu \ne 0$) discussed above.}, and a key tool in the proof is the uniform bound
\begin{equation}
\label{boundx1}
\sup_{t\ge 0} \int_S\!\rmd x\,  |x_1| \, \omega(x,t) < \infty,
\end{equation}
which follows from the conservation of energy. This estimate provides a concentration property that replaces the conservation of the moment of inertia, which in this geometry does not hold.

The case of an infinite strip with slip boundary conditions (for Euler) and no‑slip boundary conditions (for Navier–Stokes) is very challenging, due to the lack of useful conserved quantities and to the absence of the antisymmetry of $\nabla_x^\perp G(x,y)$ (as can be seen, for instance, in \cite{CapMar86, MaP94}). In fact, in the non‑viscous case, only the total vorticity mass and the energy are conserved in time. Moreover, compared with the case of the cylinder, since the Green function is uniformly bounded for $|x_1 - y_1| \ge 1$ (contrary to what happens in \cite{CD19}), the conservation of energy does not imply an estimate like Eq.~\eqref{boundx1}. Therefore, one expects that the mere assumption of non‑negative and compactly supported initial vorticity leads, in general, to the trivial upper bound $\const t$ for the diameter of the support  (which is an immediate consequence of the boundedness of the velocity field).

We also recall that vorticity confinement in the exterior of a smooth bounded domain was studied in \cite{ILL07, Mar96}, where one generally obtains a bound on the growth of the support’s diameter of order $\sqrt t$. This faster growth reflects the absence of conserved quantities that would otherwise yield a slower expansion.

It is worth emphasizing that in the Navier–Stokes setting---both in the whole plane and in the cylinder considered here---the moment of inertia is not conserved. As a consequence, obtaining bounds on the region containing almost all of the vorticity mass requires different techniques, such as estimates on higher-order moments (as in \cite{Mar94}) or an iterative method---already used in other contexts (see, for instance, \cite{BuM1, butmar2, BCM, butcavmar})---combined with the antisymmetry of $\nabla_x^\perp G(x,y)$, as employed in the present paper. We also mention \cite{Ga14, Ga15}, where the authors study the two-dimensional Navier–Stokes equations in an infinite cylinder and establish decay properties for the $L^\infty$-norms of both the velocity and vorticity fields.

Incidentally, by incorporating the estimate Eq.~\eqref{boundx1} into our approach, we obtain a slight refinement of the confinement result in the Euler case, improving upon that obtained in \cite{CD19}. In particular, we show that the size of the support grows like $(t \log t)^{1/3}$ (instead of $t^{1/3}\log^2 t$).

The plan of the paper is the following. In Section~\ref{sec:2} we present the results concerning the Navier–Stokes equations in the cylinder, while in Section~\ref{sec:3} we establish the above refinement of the confinement result for the Euler case.

\textit{A warning on the notation}. Throughout the paper, we use $C$ to denote a generic positive constant independent of time. Its numerical value may change from line to line, and it may depend on the initial data.

\section{Navier-Stokes flow}
\label{sec:2}

We assume initial data such that $\omega_0 = \partial_1 u^0_2 - \partial_2 u^0_1$ is non‑negative, bounded, and has compact support; that is, for some suitable $0 < \bar C < \infty$,
\begin{equation}
\label{Lamb0}
\Lambda_0 = \supp \omega_0 \subset \{y\in S \colon |y_1|\le  \bar C\}.
\end{equation}
By the parabolic maximum principle and mass conservation applied to Eq.~\eqref{eq:omega}, we have, for any $(x,t) \in S \times \bb R_+$,
\begin{equation}
\label{in}
0 \le \omega (x,t) \le  \|\omega_0\|_\infty, \quad  \int_S\!\rmd x\, \omega(x,t) = M_0 := \int\!\rmd x\, \omega_0(x).
\end{equation}

Defining
\begin{equation}
\label{mt}
m_t (h) = \int_{|y_1|>h}\!\rmd y\, \omega(y,t),
\end{equation}
we study the behavior of the function $m_t(h)$ for large $t$ and for $h$ suitably diverging with $t$, as established in the following theorem.

\begin{theorem}
\label{teo:mt}
Assume $\omega_0$ satisfies the assumptions stipulated above and let $m_t$ be as in Eq.~\eqref{mt}. Then:
\begin{itemize}
\item[\textit{(a)}] For each $\ell>0$ and $\alpha>1$ there exists $\mc T_{\alpha, \ell}$ such that
\begin{equation}
\label{smt}
m_t \big(\sqrt{t \log^\alpha t}\big)  \le \frac{1}{t^\ell} \quad \forall\, t\ge \mc T_{\alpha,\ell}.
\end{equation}
\item[\textit{(b)}] For any $\beta>1/2$ and $0<\delta<2\beta-1$, there exists $\mc T_{\beta, \delta}$ such that
\begin{equation}
\label{smt2}
m_t \big( t^\beta)  \le \rme^{- t^\delta} \quad \forall\, t\ge \mc T_{\beta,\delta}.
\end{equation}
\end{itemize}
\end{theorem}

\begin{proof}
Given $R\ge 2h>0$, let $W_{R,h} \colon \bb R \to [0,1]$ be a smooth function such that
\begin{equation}
\label{W0}
W_{R,h}(\xi) = \begin{cases} 1 & \text{if $|\xi|\le R$}, \\ 0 & \text{if $|\xi|\ge R+h$}, \end{cases}
\end{equation}
whose second derivative satisfies, for some positive constant $C_W$,
\begin{equation}
\label{W2}
|W_{R,h}''(\xi)| < \frac{C_W}{h^2}.
\end{equation}
In particular, for any $\xi, \eta\in \bb R$,
\begin{equation}
\label{W1}
|W_{R,h}'(\xi)-W_{R,h}'(\eta)| \le \frac{C_W}{h^2} |\xi-\eta|.
\end{equation}
Recalling that $x_1$ denotes the first component of $x = (x_1, x_2)$, we introduce the mollified version of $m_t$,
\begin{equation}
\label{mass 1}
\mu_t(R,h) = \int\! \rmd x \, \big[1-W_{R,h}(x_1)\big]\, \omega(x,t),
\end{equation}
which satisfies
\begin{equation}
\label{2mass 3}
\mu_t(R,h) \le m_t(R) \le \mu_t(R-h,h).
\end{equation}
Hence, it suffices to prove Eqs.~\eqref{smt} and \eqref{smt2} with $\mu_t$ in place of $m_t$. 

Since the function $t \mapsto \mu_t(R,h)$ is differentiable, we can compute its time derivative by means of Eq.~\eqref{Weakeq} with the test function $f(x,t) = 1 - W_{R,h}(x_1)$, obtaining
\[
\begin{split}
\frac{\rmd}{\rmd t} \mu_t(R,h) = & - \int_S\! \rmd x \, \nabla W_{R,h}(x_1) \cdot u(x,t) \, \omega(x,t) -\nu \int_S\! \rmd x \, \Delta W_{R,h}(x_1)  \, \omega(x,t) \\ 
=  & - \int_S\! \rmd x \, W_{R,h}'(x_1)  \,  u_1(x,t) \, \omega(x,t) -\nu \int_S\! \rmd x \,  W_{R,h}''(x_1)  \, \omega(x,t).
\end{split}
\]
From Eq.~\eqref{G}, we compute
\begin{equation}
\label{G_2}
\frac{\partial G}{\partial x_2}(x,y) = \frac{-\sin(x_2-y_2)}{2\left(\cosh(x_1-y_1)-\cos(x_2-y_2)\right)},
\end{equation}
which implies that $\int_{0}^{2\pi}\!\rmd y_2\, \frac{\partial G}{\partial x_2}(x,y) = 0$ for any $x \in S$ and $y_1 \in \bb R$. Therefore, using Eqs.~\eqref{vel_field} (with $m(x,t) = (0, m_2(x_1,t))$) and \eqref{hatu},
\[
u_1(x,t) = \int\!\rmd x\, \frac{\partial G}{\partial x_2}(x,y) \omega(y,t), 
\]
so that
\[
\begin{split}
\frac{\rmd}{\rmd t} \mu_t(R,h) & =  - \int_S \! \rmd x \int_S \! \rmd y\, W_{R,h}'(x_1) \frac{\partial G}{\partial x_2}(x,y)\,\omega(x,t)\, \omega(y,t) \\ & \quad -\nu \int_S\! \rmd x \,  W_{R,h}''(x_1)  \, \omega(x,t),
\end{split}
\]
Using now the antisymmetry
\begin{equation}
\label{antisim}
\frac{\partial G}{\partial x_2}(x,y) = -\frac{\partial G}{\partial x_2}(y,x),
\end{equation}
the time derivative can be written in the form
\[
\frac{\rmd}{\rmd t} \mu_t(R,h) = \int_S \! \rmd x\! \int_S \! \rmd y\, H(x,y,t) - \nu \int_S\! \rmd x \,  W_{R,h}''(x_1)  \, \omega(x,t),
\]
where
\begin{equation}
\label{H}
H(x,y,t) = - \frac 12 [W_{R,h}'(x_1)-W_{R,h}'(y_1)]\frac{\partial G}{\partial x_2}(x,y)\,\omega(x,t)\, \omega(y,t).
\end{equation}
We observe that $H(x,y,t)$ is symmetric under the exchange of $x$ and $y$ and that, by \eqref{W0}, a necessary condition for it to be nonzero is that either $|x_1|\ge R$ or $|y_1|\ge R$.  Moreover, again by Eq.~\eqref{W0}, the domain of integration in the viscous contribution to the time derivative can be restricted to the region $|x_1| \ge R$. Therefore,
\begin{align}
\label{d_tmu}
\frac{\rmd}{\rmd t} \mu_t(R,h) &= \bigg[ \int_{|x_1| \ge R}\!\rmd x\! \int_S\!\rmd y + \int_S\!\rmd x \! \int_{|y_1| > R}\!\rmd y -  \int_{|x_1| \ge R}\!\rmd x \! \int_{|y_1| > R}\!\rmd y\bigg]H(x,y,t) \nonumber \\ & \quad - \nu \int_{|x_1| \ge R}\! \rmd x \,  W_{R,h}''(x_1)  \, \omega(x,t) \nonumber  \\ & = 2 \int_{|x_1| \ge R}\!\rmd x \! \int_S\!\rmd y\,H(x,y,t)  -  \int_{|x_1| \ge R}\!\rmd x \! \int_{|y_1| \ge R}\!\rmd y\,H(x,y,t) \nonumber \\ & \quad - \nu \int_{|x_1| \ge R}\! \rmd x \,  W_{R,h}''(x_1)  \, \omega(x,t),
\end{align}
From the definition \eqref{H}, we can now estimate the right-hand side in the last equation by means of Eqs.~\eqref{W2}, \eqref{W1}, and using the definition Eq.~\eqref{mt}. We thus obtain
\begin{align}
\label{d_tmu1}
\frac{\rmd}{\rmd t} \mu_t(R,h) & \le \frac{3C_W}{2h^2} m_t(R) \|\omega_0\|_\infty \sup_{|x_1|>R} \int_S\!\rmd y\, \left|\frac{\partial G}{\partial x_2}(x,y)\right|  \, |x_1-y_1| \nonumber \\ & \quad + \nu \frac{C_W}{h^2} m_t(R).
\end{align}
From the explicit expression of $\frac{\partial G}{\partial x_2}$ in Eq.~\eqref{G_2}, one now deduces that, for suitable positive constants $C_1, C_2$, the following estimate holds:
\begin{equation}
\label{boundG}
\left|\frac{\partial G}{\partial x_2}\right| \le \frac{C_1}{|x-y|} \exp(-C_2 |x-y|).
\end{equation}
Hence, we can estimate the right-hand side in Eq.~\eqref{d_tmu1} by using Eq.~\eqref{boundG}, getting
\begin{equation}
\label{eq:mut}
\frac{\rmd}{\rmd t} \mu_t(R,h) \le \frac{C}{h^2} m_t(R).
\end{equation}
Therefore, by Eq.~\eqref{2mass 3},
\begin{equation}
\label{eq_it}
\mu_t (R,h) \le \mu_0 (R,h) +\frac{C}{h^2} \int_0^t\! \rmd s\, \mu_s (R-h, h) \quad \forall\, t>0.
\end{equation}
We are now in force to prove Eqs.~\eqref{smt} and \eqref{smt2}.

\smallskip
\textit{(a)} We assume $t$ sufficiently large and iterate inequality Eq.~\eqref{eq_it} $n=\lfloor\log t \rfloor$ times\footnote{$\lfloor a\rfloor$ denotes the integer part of $a>0$}, starting from $R_0 = \sqrt{t \log^\alpha t}$ down to $R_n = \frac 12 \sqrt{t \log^\alpha t}$, where $R_n = R_0 - n h$, and therefore $h = (2n)^{-1} \sqrt{t \log^\alpha t}$. Hence,
\[
\begin{split}
\mu_t(R_0-h,h) & \le \mu_0(R_0-h,h) + \sum_{j=1}^{n-1} \mu_0(R_{j+1},h) \,\frac{1}{j!}\left(\frac{C}{h^2} t\right)^j \\ & \quad +\frac{1}{(n-1)!} \left(\frac{C}{h^2} \right)^n \int_0^t\!{\textnormal{d}} s\,  (t-s)^{n-1}\mu_s(R_{n},h) \quad \forall\, t\gg 1.
\end{split}
\]
By assumption Eq.~\eqref{Lamb0}, we may take $t$ large enough so that $\mu_0(R_j,h)=0$ for all $j=0,\ldots,n$. Hence,
\begin{equation}
\label{mass 15'}
\mu_t(R_0-h,h) \le \frac{1}{(n-1)!} \left(\frac{C}{h^2} \right)^{n} \int_0^t\!\rmd s\, (t-s)^{n-1} \mu_s(R_{n},h) \le \frac{1}{n!} \left(\frac{C}{h^2} t \right)^n,
\end{equation}
where in the last inequality we have used the trivial bound $\mu_s(R_{n},h) \le M_0$, see Eq.~\eqref{in}. Therefore, using also Eq.~\eqref{2mass 3}, Stirling’s formula, and $n=\lfloor\log t\rfloor$, we obtain
\[
m_t(R_0) \le \mu_t(R_0 -h,h) \le  \left( \frac{C  t \log^2 t}{t \log^{1+\alpha}t}\right)^{\log t},
\]
which yields Eq.~\eqref{smt} upon choosing $\mc T_{\alpha,\ell}$ sufficiently large.

\smallskip
\textit{(b)} As before, we assume $t$ sufficiently large and iterate inequality Eq.~\eqref{eq_it} $n=\lfloor t^\delta \rfloor$ times, from $R_0 = t^\beta$ down to $R_n = \frac{1}{2} t^\beta$, where $R_n = R_0 - n h$, and hence $h = (2n)^{-1} t^\beta$. Repeating the same argument as above (details omitted), we obtain
\[
m_t(R_0)\le \left(  \frac{C t}{t^{2(\beta-\delta)}t^\delta} \right)^{t^\delta}
=\left(  \frac{C }{t^{2\beta-\delta-1}} \right)^{t^\delta},
\]
from which Eq.~\eqref{smt2} follows upon choosing $\mc T_{\beta,\ell}$ sufficiently large.
\end{proof}

We remark that in \cite{Mar94} a result analogous to item (b) of Theorem~\ref{teo:mt} was obtained by using higher moments, whereas here we rely on an iterative argument together with property Eq.~\eqref{antisim}, which allows us to establish both items (a) and (b) of Theorem~\ref{teo:mt}.

\section{Euler flow}
\label{sec:3}

In the absence of viscosity the vorticity is transported (see Eq.~\eqref{eq:omega} for $\nu=0$), so that
\begin{equation}
\label{Cons1}
\omega(x(x_0,t),t) = \omega_0(x_0),
\end{equation}
where $x(x_0,t)$ is the trajectory of the fluid particle initially at $x_0$, i.e.,
\begin{equation}
\label{Cons2}
\dot x(x_0,t) = u(x(x_0,t),t), \quad x(x_0,0) = x_0,
\end{equation}
$\dot x(x_0,t) := \frac{\partial}{\partial t} x(x_0,t)$. As in the previous section, we assume that $\omega_0$ is non‑negative, bounded, and satisfies Eq.~\eqref{Lamb0}. Clearly, as $\omega(x,t)$ is simply transported, Eq.~\eqref{in} also holds in the present setting. Regarding the velocity field, following \cite{CD19}, we choose
\begin{equation}
\label{ueuler}
u(x,t) = \int_S\! \rmd y\, \nabla_x^\perp G(x,y)\,\omega(y,t),
\end{equation}
which, under our assumption of bounded and compactly supported vorticity, yields the following behavior at infinity,
\[
\lim_{x_1 \to \pm \infty} u_1(x,t) = 0, \qquad \lim_{x_1\to - \infty} u_2(x,t) = - \lim_{x_1\to + \infty} u_2(x,t).
\]

In this case, we obtain the following nontrivial bound on the growth in time of the diameter of the support: for any $\alpha>1$, it does not exceed $C(t \log^\alpha t)^{1/3}$ for $t \gg 1$. This is proved by means of the iterative method already used in the previous section, here enforced by incorporating the concentration property Eq.~\eqref{boundx1}, see Lemma \ref{lem:mt} below. In particular, this approach shows that the vorticity mass lying farther than $C(t \log^\alpha t)^{1/3}$ from the origin (around which the initial vorticity is supported) is extremely small. We then conclude by showing that the velocity field generated by this negligible mass cannot push vortex filaments that have reached such distances any further away. We are now in a position to state our result.

\begin{theorem}
\label{th_d}
Let $\omega(x, t)$ be the solution to Eqs.~\eqref{Cons1}, \eqref{Cons2}, and \eqref{ueuler} with $\omega_0$ as stipulated above. Define
\[
d_\omega (t) := \sup_{x, y \in \Lambda(t)} |x_1-y_1|,
\]
where $\Lambda(t) :=  \supp\, \omega (\cdot,t)$. Then
\begin{equation}
\label{size}
\lim_{t\to \infty} (t\log^\alpha t)^{-1/3} d_\omega (t) = 0 \quad \forall\,\alpha >1.
\end{equation}
\end{theorem}

\subsection{Proof of Theorem \ref{th_d}}
\label{sec:3.1}

The proof is organized into two preliminary lemmata and the proof of Eq.~\eqref{size}. The first lemma provides the key estimates on the horizontal velocity of fluid particles which are furthest from the origin.

\begin{lemma}
\label{lem:3}
Define
\begin{equation}
\label{Rt}
R_t := \max\{|x_1| \colon x\in \Lambda(t) \}.
\end{equation}
Given $x_0\in\Lambda(0)$, let $x(x_0,t)$ be the solution to Eq.~\eqref{Cons2} and suppose at time $t$ it happens that
\begin{equation}
\label{hstimv}
|x_1(x_0,t)| = R_t>0.
\end{equation}
Then, at this time $t$,
\begin{equation}
\label{stimv}
|\dot x_1(x_0,t)| \le \frac{2C_1\rme^{-C_2 R_t/2}}{R_t} + C \sqrt{m_t(R_t/2)},
\end{equation}
with $C_1,C_2$ as in Eq.~\eqref{boundG} and $m_t(\cdot)$ as in Eq.~\eqref{mt}.
\end{lemma}

\begin{proof}
Letting $x=x(x_0,t)$, by Eqs.~\eqref{Cons1}, \eqref{Cons2},  \eqref{ueuler}, and \eqref{boundG},
\begin{equation}
\label{eq:u1}
|\dot x_1(x_0,t)| \le |u_1(x,t)| \le \int_{\Lambda(t)}\! \rmd y\, \frac{C_1}{|x-y|} \exp(-C_2|x-y|)\, \omega(y,t).
\end{equation}
We split the integration region into two parts, the set $A_1= \{y\in \Lambda(t) \colon |y_1 |\le R_t/2\}$ and the set $A_2 = \{y\in \Lambda(t) \colon  R_t/2 < |y_1| \le R_t\}$. Then,
\begin{equation}
\label{distance1}
|\dot x_1(x_0,t)| \le H_1 + H_2,
\end{equation}
where
\begin{equation}
\label{in A_1}
H_1 = \int_{A_1}\! \rmd y\, \frac{C_1}{|x-y|} \exp(-C_2|x-y|)\, \omega(y,t)
\end{equation}
and
\begin{equation}
\label{in A_2}
H_2 = \int_{A_2}\! \rmd y\, \frac{C_1}{|x-y|} \exp(-C_2|x-y|)\, \omega(y,t).
\end{equation}
To estimate $H_1$, we notice that, in view of Eq.~\eqref{hstimv}, $|x-y|\ge|x_1-y_1|\ge R_t/2$, therefore
\begin{equation}
\label{in H_11}
H_1 \le \frac{2C_1}{R_t} \rme^{-C_2 R_t/2}.
\end{equation}
On the other hand,
\[
H_2 \le C_1 \int_{A_2}\! \rmd y\, \frac 1{|x-y|} \, \omega(y,t)
\]
and we notice that the function $|x-y|^{-1}$ diverges monotonically as $y\to x$, so that the maximum of the integral in the right-hand side is achieved when we rearrange the vorticity mass as close as possible to the singularity. By the bound on $\omega$ in Eq.~\eqref{in} and since, by Eq.~\eqref{mt}, $m_t(R_t/2)$ is equal to the total amount of vorticity in $A_2$, this rearrangement gives
\begin{equation}
\label{h2}
H_2 \le C_1 \int_{|y| \le \varrho}\!\rmd y\, \frac{1}{|y|} = 2\pi C_1 \varrho,
\end{equation}
where the radius $\varrho$ is such that $\pi \varrho^2\|\omega_0\|_\infty = m_t(R_t/2)$. The estimate Eq.~\eqref{stimv} now follows from Eqs.~\eqref{distance1}, \eqref{in H_11}, and \eqref{h2}.
\end{proof}

\begin{lemma}
\label{lem:mt}
Under the assumptions of Theorem \ref{th_d} and letting $m_t$ as in Eq.~\eqref{mt}, for each $\ell>0$ and $\alpha>1$ there exists $\mc T_{\alpha, \ell}$ such that
\begin{equation}
\label{smte}
m_t \big( (t \log^\alpha t)^{1/3} \big)  \le \frac{1}{t^\ell} \quad \forall\, t\ge \mc T_{\alpha,\ell}.
\end{equation}
\end{lemma}

\begin{proof}
By arguing as in the proof of Theorem \ref{teo:mt}, we arrive at Eq.~\eqref{d_tmu} without the last term in the right-hand side (as $\nu =0$), i.e., 
\[
\frac{\rmd}{\rmd t} \mu_t(R,h) =2 \int_{|x_1| \ge R}\!\rmd x \! \int_S\!\rmd y\,H(x,y,t)  -  \int_{|x_1| \ge R}\!\rmd x \! \int_{|y_1| \ge R}\!\rmd y\,H(x,y,t),
\]
with $H(x,y,t)$ as in Eq.~\eqref{H}. To exploit the additional concentration property Eq.~\eqref{boundx1}, we further decompose the right-hand side in the above identity, getting
\[
\frac{\rmd}{\rmd t} \mu_t(R,h) = A_1 + A_2 + A_3,
\]
with
\[
\begin{split}
A_1 & = 2 \int_{|x_1| > R}\!\rmd x \! \int_{|y_1| \le R/2}\!\rmd y\,H(x,y,t), \\ A_2 & = 2 \int_{|x_1| > R}\!\rmd x \! \int_{|y_1| > R/2}\!\rmd y\, H(x,y,t), \\ A_3 & = -  \int_{|x_1| > R}\! \rmd x \! \int_{|y_1| > R}\!\rmd y \, H(x,y,t).
\end{split}
\]
The term $A_1$ becomes negligible for large $R$ due to the decay of the kernel Eq.~\eqref{boundG}; more precisely, using also Eq.~\eqref{W1}, we get
\[
\begin{split}
|A_1| & \le \frac{C_W}{h^2} m_t(R) \sup_{|x_1|>R} \int_{|y_1| \le R/2}\!\rmd y\, \Big|\frac{\partial G_2}{\partial x_2}(x,y)\Big|\, |x_1-y_1| \, \omega(y,t) \\ & \le \frac{C_W}{h^2} m_t(R) \sup_{|x_1|>R} \int_{|y_1| \le R/2}\!\rmd y\, C_1\exp(-C_2 |x-y|)\, \omega(y,t) \\ & \le \frac{C}{h^2} \exp(-C_2 R/2)m_t(R).
\end{split}
\]
To estimate $A_2$ and $A_3$, we first apply Eq.~\eqref{W1} and then we multiply and divide by $|y_1|$, gaining an extra factor $1/R$,
\[
\begin{split}
|A_2| + |A_3| & \le 6 \frac{C_W}{Rh^2} m_t(R) \int_{|y_1| > R/2}\!\rmd y\, \Big|\frac{\partial G_2}{\partial x_2}(x,y)\Big| \, |x_1-y_1| \, |y_1|\, \omega(y,t) \\ & \le 6C_1 \frac{C_W}{Rh^2} m_t(R) \int\!\rmd y\,  |y_1| \, \omega(y,t) \le \frac{C}{Rh^2},
\end{split}
\]
where we used Eq.~\eqref{boundx1} in the inequality. Therefore,
\[
\frac{\rmd}{\rmd t} \mu_t(R,h) \le \frac{C}{R h^2} m_t(R),
\]
whence
\[
\mu_t (R,h) \le \mu_0 (R,h) +\frac{C}{Rh^2} \int_0^t\! \rmd s\, \mu_s (R-h, h) \quad \forall\, t>0.
\]
Analogously to what done in the proof of item (a) of Theorem~\ref{teo:mt}, we assume $t$ sufficiently large and iterate the above inequality $n=\lfloor\log t \rfloor$ times, starting from $R_0 =  (t \log^\alpha t)^{1/3}$ down to $R_n = \frac 12  (t \log^\alpha t)^{1/3}$, where $R_n = R_0 - n h$, and therefore $h = (2n)^{-1} (t \log^\alpha t)^{1/3}$. Following the same steps as in the proof of Theorem~\ref{teo:mt} (details omitted) leads again to the estimate
\[
m_t(R_0) \le  \left( \frac{C  t \log^2 t}{t \log^{1+\alpha}t}\right)^{\log t},
\]
which yields Eq.~\eqref{smte}  upon choosing $\mc T_{\alpha,\ell}$ sufficiently large.
\end{proof}

\begin{proof}[Proof of Eq.~\eqref{size}]
From the definition Eq.~\eqref{Rt}, it is enough to prove that
\begin{equation}
\label{size1}
\lim_{t\to \infty} (t\log^\alpha t)^{-1/3} R_t = 0 \quad \forall\,\alpha >1.
\end{equation}
To this end, we proceed by contradiction and assume that there exist $\alpha>1$, $Q>0$, and a sequence of positive times $\{T_j\}_{j\in \bb N}$, diverging to infinity as $j\to\infty$, such that
\begin{equation}
\label{eq:abs}
R_{T_j} \ge 3Q \big[(1+T_j) \log^\alpha (2+T_j) \big]^{1/3} \quad \forall\, j\in \bb N.
\end{equation}

We recall that, in view of definition Eq.~\eqref{Rt} and Lemma \ref{lem:3}, for any $x_0\in \Lambda(0)$ and $t>0$ we have $|x_1(x_0,t)| \le R_t$, and whenever $|x_1(x_0,t)|=R_t$ the differential inequality Eq.~\eqref{stimv} holds. We claim that this implies 
\begin{equation} 
\label{suppRt} 
\Lambda (t)  \subset \{x \in S \colon |x_1|< \rho(t) \} \quad \forall\, t \in [s_0,s_1]  \quad \forall\, [s_0,s_1]  \subset [0,\infty),  
\end{equation}
where $\rho(t)$ solves 
\begin{equation} 
\label{eqdiffRt} 
\dot{\rho}(t) = \frac{4C_1\rme^{-C_2 \rho(t)/2}}{\rho(t)} + g(t),
\end{equation}
with initial datum $\rho(s_0) > R_{s_0}$ and $g(t)$ any smooth function which is an upper bound for the last term in Eq.~\eqref{stimv}. Indeed, $|x_1| < \rho(s_0)$ for any $x\in \Lambda(s_0)$, and by contradiction, if there were a first time $t_* \in (s_0,s_1]$ such that $x_1(x_0,t_*) = \rho(t_*)$ [resp.~$x_1(x_0,t_*) = -\rho(t_*)$] for some $x_0 \in \Lambda(0)$, then necessarily $\rho(t_*) = R_{t_*}$. Hence, by Eqs.~\eqref{stimv} and \eqref{eqdiffRt}, $\dot{\rho}(t_*)$ would be strictly larger than $\dot{x}_1(x_0,t_*)$ [resp. $-\dot{\rho}(t_*)$ strictly smaller than $\dot{x}_1(x_0,t_*)$], contradicting the fact that $t_*$ is the first time at which the graph of $t \mapsto x_1(x_0,t)$ crosses that of $t \mapsto \rho(t)$ [resp. $t \mapsto -\rho(t)$].

Fix now $\alpha' \in (1, \alpha)$ with $\alpha$ as in Eq.~\eqref{eq:abs} and, recalling Eqs.~\eqref{Lamb0} and \eqref{smte}, assume $T_j$  (i.e., $j \in \bb N$) so large that
\begin{equation}
\label{eq:condj}
\bar C < 2\big[T_j \log^{\alpha'}T_j\big]^{1/3} < Q \big[(1+T_j) \log^\alpha (2+T_j)\big]^{1/3} , \quad T_j^{1/4} > \mc T_{\alpha',4}
\end{equation}
Let then
\[
t_0 = \sup\big\{t >0 \colon R_t < Q \big[(1+T_j) \log^\alpha (2+T_j)\big]^{1/3}\big\}.
\]
From Eq.~\eqref{eq:abs} and the first inequality in Eq.~\eqref{eq:condj}, $t_0\in (0, T_j)$. We next analyze separately the cases $t_0 \in (0,T_j^{1/4})$ and $t_0 \in [T_j^{1/4},T_j)$, showing that in both cases we get an absurd as soon as $T_j$ is sufficiently large.

(i) If $t_0 \in (0,T_j^{1/4})$, from Eq.~\eqref{eq:u1} and the same rearrangement argument used to prove Eq.~\eqref{h2} we deduce
\[
|u_1(x,t)| \le \int_{\Lambda(t)}\! \rmd y\, \frac{C_1}{|x-y|}\, \omega(y,t) \le C,
\]
so that, by the assumption Eq.~\eqref{Lamb0} and the definition of $t_0$, 
\[
Q \big[(1+T_j) \log^\alpha (2+T_j)\big]^{1/3} - \bar C \le R_{t_0} - R_0 \le Ct_0  \le C T_j^{1/4},
\]
which cannot be true for $T_j$ large.

(ii) If $t_0 \in [T_j^{1/4},T_j)$, let $\rho(t)$ be as in Eq.~\eqref{eqdiffRt}, with $[s_0,s_1] = [t_0,T_j]$ and
\[
\rho(t_0) = 2Q \big[(1+T_j) \log^\alpha (2+T_j)\big]^{1/3}, \quad g(t) \le \frac{C}{t^2} \quad \forall\, t\in [t_0,T_j],
\]
with $g(t)$ chosen appropriately. By the definition of $t_0$ and the second inequality in Eq.~\eqref{eq:condj}, it follows that $\rho(t_0) > R_{t_0}$ and $R_t \ge 2\big[T_j \log^{\alpha'}T_j\big]^{1/3} $ for any $t\in [t_0,T_j]$. Therefore, by the third inequality in Eq.~\eqref{eq:condj} and Eq.~\eqref{smte}, the last term in Eq.~\eqref{stimv} is bounded by $C/t^2$ for any $t\in [t_0,T_j]$. This guarantees that $g(t)$ can be chosen satisfying also the requirement of being an upper bound to the last term in Eq.~\eqref{stimv}. As a consequence, Eq.~\eqref{suppRt} holds true for any $t\in [t_0,T_j]$. In particular, $\rho(t) \ge R_t \ge Q  \big[(1+T_j) \log^\alpha (2+T_j)\big]^{1/3}$ for any $t\in [t_0,T_j]$, from which we deduce that, using also Eqs.~\eqref{eq:abs} and \eqref{eqdiffRt},
\[
\begin{split}
& Q  \big[(1+T_j) \log^\alpha (2+T_j)\big]^{1/3} \le R_{T_j} - \rho(t_0) \le \rho(T_j) - \rho(t_0) \nonumber \\ & = \int_{t_0}^{T_j}\!\rmd t\, \left[ \frac{4C_1\rme^{-\frac 12C_2\rho(t)}}{\rho(t)} + g(t) \right] \le  \frac{4C_1\rme^{- \frac 12C_2 Q  \big[(1+T_j) \log^\alpha (2+T_j)\big]^{1/3}}}{Q  \big[(1+T_j) \log^\alpha (2+T_j)\big]^{1/3}}T_j + \frac{C}{T_j^{1/4}},
\end{split}
\]
which again leads to a contradiction for $T_j$ large enough.
\end{proof}

\subsection*{Statements and declarations}

\begin{itemize}
\item \textbf{Funding} No funding was received for conducting this study.
\item \textbf{Conflict of interest} The authors have no competing interests to declare that are relevant to the content of this article.
\item \textbf{Data Availability Statement} Data sharing not applicable to this article as no datasets were generated or analysed during the current study.
\end{itemize}

\subsection*{Acknowledgments}
Work performed under the auspices of GNFM-INDAM and the Italian Ministry of the University (MUR).

\end{document}